\documentclass[11pt]{amsart}

\usepackage[a4paper,vmargin=4cm]{geometry} 
\usepackage{amsfonts,amssymb,amscd,amstext} 
\usepackage{graphicx} 
\usepackage{color} 
\usepackage[dvips]{epsfig} 

\usepackage[utf8]{inputenc} 
\usepackage{hyperref} 

\usepackage{verbatim} 
\usepackage{fancyhdr} 
\input xy 
\xyoption{all} 

\usepackage{times} 
\usepackage{enumerate} 
\usepackage{titlesec} 
\usepackage{mathrsfs} 

\numberwithin{equation}{section} 
\numberwithin{figure}{section}

\titleformat{\section}%[display] 
{\filcenter\bfseries\large} {\thesection{.}}{0.2cm}{}%[$\vspace*{-1.0cm}$] 
\titleformat{\subsection}[runin] 
{\bfseries} {\thesubsection{.}}{0.15cm}{}[.] 
\titleformat{\subsubsection}[runin] 
{\em}{\thesubsubsection{.}}{0.15cm}{}[.] 
\usepackage[up,bf]{caption} 
\newtheorem{theorem}{Theorem}[section]

\newtheorem{corollary}[theorem]{Corollary} 

\theoremstyle{definition}

\newcommand\Oscr{\mathscr{O}}

\newcommand\B{\mathbb{B}} 
\newcommand\C{\mathbb{C}} 
 
\newcommand\CP{\mathbb{CP}}

\newcommand\N{\mathbb{N}} 
 
\newcommand\R{\mathbb{R}}

\newcommand\igot{\mathfrak{i}}

\renewcommand\igot{\mathfrak{i}}

\renewcommand\imath{\igot}

\newcommand\hra{\hookrightarrow}

\def\Ell1{\mathrm{Ell_1}} 
\def\CEll1{\mathrm{CEll_1}}

\begin{document} 

\fancyhead[LO]{Proper holomorphic embeddings with small limit sets} 
\fancyhead[RE]{F.\ Forstneri\v c} 
\fancyhead[RO,LE]{\thepage} 

\thispagestyle{empty} 

%% Title 
%\vspace*{5mm} 

\begin{center} 
{\bf \LARGE Proper holomorphic embeddings with small limit sets} 

\vspace*{0.5cm} 

%% Authors 
{\large\bf Franc Forstneri\v c} 
\end{center} 

\vspace*{0.5cm} 

{\small 
\noindent {\bf Abstract} 
\ Let $X$ be a Stein manifold of dimension $n\ge 1$.
Given a continuous positive increasing function 
$h$ on $\R_+=[0,\infty)$ with $\lim_{t\to\infty} h(t)=\infty$,
we construct a proper holomorphic embedding
$f=(z,w):X\hra \C^{n+1}\times \C^n$ satisfying $|w(x)|<h(|z(x)|)$ 
for all $x\in X$. In particular, $f$ may be chosen such that its limit set at 
infinity is a linearly embedded copy of $\CP^n$ in $\CP^{2n}$.
%which is the smallest possible dimension of the limit set unless $X$ is algebraic.
\hspace*{0.1cm} 
}

\noindent{\bf Keywords:}\hspace*{0.1cm} 
Stein manifold, proper holomorphic embedding

\noindent{\bf MSC (2020):}\hspace*{0.1cm} 
Primary 32H02; Secondary 32E10, 32Q56

%
%  32B15: (1973-now) Analytic subsets of affine space
%  32E10: Stein manifolds
%  32E30: (1973-now) Holomorphic and polynomial approximation, Runge pairs, interpolation
%  32H02: (1991-now) Holomorphic mappings, (holomorphic) embeddings and related questions
%  32M17: Automorphism groups of Cn and affine manifolds 
%  32Q56: Oka principle and Oka manifolds 
%
%  52A20 (1973-now) Convex sets in n dimensions (including convex hypersurfaces) [See also 53A07, 53C45] 
%  52A27 (1991-now) Approximation by convex sets 
%

\noindent{\bf Date:}\hspace*{0.1cm} 
30 November 2023

%%%%%%%%%% 
%%%%%%%%%% 
%%%%%%%%%% 
%%%%%%%%%% 
%%%%%%%%%% 
%%%%%%%%%% 

% 
% INTRODUCTION 
% 

\section{The main result}\label{sec:intro} 
A theorem of Remmert \cite{Remmert1956CR}, Narasimhan 
\cite{Narasimhan1960AJM}, and Bishop \cite{Bishop1961AJM} states 
that every Stein manifold $X$ of dimension $n\ge 1$ admits
a proper holomorphic map to $\C^{n+1}$, a proper holomorphic immersion
to $\C^{2n}$, and a proper holomorphic embedding in $\C^{2n+1}$.
(See also \cite[Chap.\ VII.C]{GunningRossi2009}.)  
We are interested in the question how much space proper holomorphic 
embeddings or immersions $X\to \C^N$ need, and how small can 
their limit sets at infinity be. 

By Remmert \cite{Remmert1956MA}, the image 
$A=f(X)\subset \C^N$ of a proper holomorphic map $f:X\to\C^N$ is a closed 
complex subvariety of pure dimension $n=\dim X$.
% It was shown by Rudin \cite{Rudin1968} 
% (see also Chirka \cite[Theorem 3,\ p.\ 78]{Chirka1989}) 
Such an $A$ is algebraic if and only if it is contained, 
after a unitary change of coordinates on $\C^N$, in a domain of the form 
\[
	D=\{(z,w)\in \C^n\times \C^p = \C^N: |w|< C(1+|z|)\}
\]
for some $C>0$ (see Chirka \cite[Theorem 2,\ p.\ 77]{Chirka1989}).
% and the related result of Rudin \cite{Rudin1968}, \cite[Theorem 3,\ p.\ 78]{Chirka1989}).
%This can equivalently be expressed as follows.
%
Equivalently, if $H=\CP^N\setminus \C^N\cong \CP^{N-1}$ denotes 
the hyperplane at infinity and $A_\infty=\overline A\, \cap H$,
where $\overline A$ is the topological closure of $A$ in $\CP^N$,
then $A$ is algebraic if and only if there is a linear subspace 
% Franc: correction on the next line, CP^N -
$L\cong \CP^{N-n-1}$ of $H\cong \CP^{N-1}$ such that $L\cap A_\infty =\varnothing$. 
If this holds then $\overline A$ and $A_\infty$ are algebraic subvarieties of pure 
dimension $n$ and $n-1$, respectively.
If $X$ is not algebraic then the image of any proper holomorphic 
immersion $f:X\to\C^N$ is not algebraic either, so its limit set 
$f(X)_\infty \subset \CP^{N-1}$ has a nonempty intersection with 
every linear subspace $\CP^{N-n-1}\cong L \subset \CP^{N-1}$.
 
We construct proper holomorphic embeddings with 
images in small Hartogs domains. 

%
%  MAIN THEOREM
%
\begin{theorem}\label{th:main1} 
Let $X$ be a Stein manifold of dimension $n\ge 1$. 
Given a continuous increasing function $h:[0,\infty)\to (0,\infty)$ 
with $\lim_{t\to\infty} h(t)=\infty$ there 
exist a proper holomorphic embedding $(z,w):X\hra \C^{n+1}\times \C^n$ and 
a proper holomorphic immersion $(z,w):X\to \C^{n+1}\times \C^{n-1}$ satisfying 
\begin{equation}\label{eq:estimate}
	\text{$|w(x)|<h(|z(x)|)$\ \ for all $x\in X$.}
\end{equation}
% Franc   Change on the next line
Furthermore, given a compact $\Oscr(X)$-convex set $K$ in $X$,
an open neighbourhood $U\subset X$ of $K$, and a 
holomorphic map $f_0=(z_0,w_0): U \to \C^{n+1}\times \C^p$ 
satisfying \eqref{eq:estimate} for all $x\in K$, we can approximate $f_0$ 
uniformly on $K$ by a proper holomorphic embedding 
$f=(z,w):X\to \C^{n+1}\times \C^p$ if $p\ge n$, resp.\ immersion if $p=n-1$, 
satisfying \eqref{eq:estimate}. 
\end{theorem}

% In view of the above discussion, this result is optimal with respect to the dimension of the first coordinate subspace.

The function $h$ in Theorem \ref{th:main1} can be chosen to grow arbitrarily slowly, 
% Franc: change on the next line
and hence the image $f(X)$ may be arbitrarily close to the subspace $\C^{n+1}\times \{0\}^n$ 
in the Fubini--Study metric on $\CP^{2n+1}$. 
Choosing $h$ such that $\lim_{t\to\infty} h(t)/t=0$ gives the following corollary. 

\begin{corollary}\label{cor:main1}
Every Stein manifold $X$ of dimension $n\ge 1$ admits a proper
holomorphic embedding $f:X\hra \C^{2n+1}$ whose limit set
$f(X)_\infty = \overline{f(X)} \cap H$ is a linearly embedded copy of 
$\CP^n$ in $H=\CP^{2n+1}\setminus \C^{2n+1} \cong \CP^{2n}$. 
In particular, every open Riemann surface
$X$ admits a proper holomorphic embedding in $\C^3$
whose limit set is a projective line $\CP^1\subset \CP^2$.
The analogous result holds for proper holomorphic immersions
$X\to \C^{2n}$. 
\end{corollary}

By the preceding discussion, the limit set $f(X)_\infty$ intersects every projective
subspace $L\subset\CP^{N-1}$ of dimension $N-n-1$, unless $f(X)$
is algebraic. Therefore, the nonalgebraic embeddings given by 
Corollary \ref{cor:main1} have the smallest possible limit sets.

Given a nonalgebraic complex subvariety $X$ of $\C^N$, its closure
$\overline X\subset\CP^N$ and the limit set $X_\infty \subset\CP^{N-1}$ 
need not be analytic subvarieties, and for any pair of integers 
$1\le n<N$ there are $n$-dimensional closed complex 
submanifolds $X\subset\C^N$ with $X_\infty = \CP^{N-1}$. 
(This always holds if $N=n+1$ and $X$ is nonalgebraic.) 
Indeed, if $X$ is a closed complex subvariety of $\C^N$ $(N>1)$ 
then for any closed discrete set $B=\{b_j\}_{j\in\N} \subset\C^N$ 
there exist a domain $\Omega\subset \C^N$ containing $X$ and a 
biholomorphic map $\Phi:\Omega \to \C^N$ such that 
$B \subset \Phi(X)$ (see \cite[Theorem 6.1]{Forstneric1999JGEA}
or \cite[Theorem 4.17.1 (i)]{Forstneric2017E}). Note that $X'=\Phi(X)$ 
is a closed complex subvariety of $\C^N$. Choosing 
$B$ such that its closure in $\CP^N$ contains the hyperplane at infinity 
implies $X'_\infty =\CP^{N-1}$. A characterization 
of the closed subsets of $\CP^{N-1}$ which are limit sets of closed complex
subvarieties of $\C^N$ of a given dimension does not seem to be known.

The corollary is especially interesting in dimension $n=1$. A long standing open
question (the Forster conjecture \cite{Forster1970}, also called  
the Bell--Narasimhan conjecture \cite{BellNarasimhan1990,Belletal1998})
asks whether every open Riemann surface, $X$, admits a proper holomorphic 
embedding in $\C^2$. Recent surveys of this subject can be 
found in \cite[Secs.\ 9.10--9.11]{Forstneric2017E} and the 
preprint \cite{AlarconLopez2022} by Alarc\'on and L\'opez, where 
the authors constructed a proper harmonic embedding of any open
Riemann surface in $\C\times \R^2\cong \C^2$ with a holomorphic 
first coordinate function. Note that if $X\to \C^2$ is a proper holomorphic
map with nonalgebraic image then $f(X)_\infty=\CP^1$. 
(There are algebraic open Riemann surfaces which do not embed as 
smooth proper affine curves in $\C^2$.) 
Corollary \ref{cor:main1} gives proper holomorphic embeddings $f:X\hra\C^3$ 
whose images are arbitrarily close to the subspace $\C^2\times \{0\}$ 
in the Fubini--Study metric on $\CP^3$, and $f(X)_\infty = \CP^1$.

It was recently shown by Drinovec Drnov\v sek and Forstneri\v c 
\cite[Theorem 1.3]{DrinovecForstneric2023JGA} that, under a mild condition
on an unbounded closed convex set $E\subset \C^{N}$,
proper holomorphic embeddings $f:X\hra \C^{N}$ from 
any Stein manifold $X$ with $2\dim X<N$ 
such that $f(X)\subset \Omega=\C^{N}\setminus E$ are dense in the space
$\Oscr(X,\Omega)$ of all holomorphic maps $X\to\Omega$.
A similar result holds for immersions if $2\dim X\le N$.
%This holds in particular if there is a real hyperplane $H\subset\C^{N}$
%such that $E$ lies in one of the closed half-spaces determined by $H$ 
%and $E\cap H$ is a nonempty bounded set.
Their proof relies on the fact, proved by Forstneri\v c  
and Wold \cite{ForstnericWold2023IMRN}, that such $\Omega$ is an 
Oka domain. (See \cite[Definition 5.4.1 and Theorem 5.4.4]{Forstneric2017E}
for the definition and the main result concerning Oka manifolds.) 
Note that the domains in Theorem \ref{th:main1} are much smaller than 
those in \cite[Theorem 1.3]{DrinovecForstneric2023JGA} when the codimension
is at least $2$. On the other hand, Theorem \ref{th:main1} does not 
pertain to proper maps in codimension 1 (the case $p=0$).
We do not know whether a Hartogs domain of the form 
\begin{equation}\label{eq:Omega}
	\Omega=\{(z,w)\in \C^{n+1}\times \C^p: |w| < h(|z|)\},\quad 
	n\ge 1,\ p\ge 1, 
\end{equation}
which appears in Theorem \ref{th:main1}, is an Oka domain, 
except if $p=1$, the function $h>0$ on $\R_+$ 
grows at least linearly at infinity, and $\log h(|z|)$ is plurisubharmonic on $\C^{n+1}$ 
(see Forstneri\v c and Kusakabe \cite[Proposition 3.1]{ForstnericKusakabe2023X}). 
Our proof does not require the Oka property of $\Omega$ in \eqref{eq:Omega}.

We mention that a Stein manifold of dimension $n\ge 2$ admits a 
proper holomorphic embedding $X\hra\C^{N}$ with $N=\left[\frac{3n}{2}\right]+1$
and a proper holomorphic immersion with $N=\left[\frac{3n+1}{2}\right]$; see 
Eliashberg and Gromov \cite{EliashbergGromov1992}, Sch\"urmann 
\cite{Schurmann1997}, and \cite[Theorem 9.3.1]{Forstneric2017E}. 
The proofs are very delicate and depend on Oka theory. 
We do not know whether one can expect a similar control of the range of the
embedding in these dimensions.

%
% 	PROOF
%
\section{Proof of Theorem \ref{th:main1}}\label{sec:proof}

Our proof of Theorem \ref{th:main1} relies on the following technical result,
which is a special case of \cite[Theorem 1.1]{DrinovecForstneric2010AJM}
by Drinovec Drnov\v sek and Forstneri\v c. 
(See also \cite[Theorem 6]{ForstnericRitter2014}, which is based
on the same result.) % by Forstneri\v c and Ritter.
Similar results were obtained earlier by Dor \cite{Dor1993,Dor1995}.

%
%   THEOREM 6 FROM FORSTNERIC-RITTER 2014 MZ
%
\begin{theorem} \label{th:proper}
Assume that $X$ is a Stein manifold of dimension $n\ge 1$, 
$D$ is a relatively compact, smoothly bounded, strongly pseudoconvex domain 
in $X$, $K$ is a compact set contained in $D$, $t_0$ is a real number, 
$\sigma: \C^{n+1}\to\R$ is a strongly plurisubharmonic exhaustion 
function which has no critical points in the set $\{\sigma\ge t_0\}$, 
and $g_0 : \overline D\to \C^{n+1}$ is a continuous map 
that is holomorphic in $D$ and satisfies
$g_0\bigl(\overline {D\setminus K}\bigr) \subset \{\sigma>t_0\}$. 
Given numbers $t_1>t_0$ and $\epsilon>0$, there is a 
holomorphic map $g : \overline D\to \C^{n+1}$ satisfying 
the following conditions:
\begin{enumerate}[\rm (a)]
%\item $g(\overline{D\setminus K}) \subset \{\sigma>t\}$.
\item $g(bD)\subset \{\sigma>t_1\}$.
\item $\sigma(g(x))>\sigma(g_0(x)) -\epsilon$ for all $x\in \overline D$. 
\item $|g(x)-g_0(x)|<\epsilon$ for all $x\in K$.
\end{enumerate}
%It can also be approximated uniformly on $K$ by proper holomorphic maps 
%$g: D\to \C^{n+1}$ satisfying $g(D\setminus \mathring K) \subset \{\sigma>t\}$.  
\end{theorem}

% F:  naslednji pogoj lahko dam v to?ko (b) izreka.
Note if $\epsilon>0$ is small enough then condition (b) implies
\[ %begin{equation}\label{eq:g}
	g\bigl(\overline {D\setminus K}\bigr) \subset \{\sigma>t_0\}.
\] %end{equation}
The analogous result holds much more generally,
and we only stated the case that will be used here. For condition (b),
see \cite[Lemma 5.3]{DrinovecForstneric2010AJM}, which is the main 
inductive step in \cite[proof of Theorem 1.1]{DrinovecForstneric2010AJM}.
We remark that a map from a compact set in a complex manifold is 
said to be holomorphic if it is holomorphic in an open neighbourhood of
the said set.

\begin{proof}[Proof of Theorem \ref{th:main1}]
We shall construct proper holomorphic embeddings $X\hra\C^{N}$ 
with $N\ge 2n+1$ satisfying \eqref{eq:estimate}; the same arguments 
will yield immersions when $N= 2n$. 

Let $\Omega\subset \C^{N}$ be a domain
of the form \eqref{eq:Omega} with coordinates $(z,w)\in \C^{n+1}\times \C^p$ 
where $p\ge n$, $N=n+1+p$, and the function $h:[0,\infty)\to (0,\infty)$ 
is as in the theorem. We shall use Theorem \ref{th:proper} with the exhaustion function 
$\sigma(z)=|z|$ on $\C^{n+1}$; the nonsmooth point at the origin will not matter. 
We denote by $\B$ the open unit ball in $\C^{n+1}$.

Since the set $K\subset X$ is compact and $\Oscr(X)$-convex,
there exist a smooth strongly plurisubharmonic Morse 
exhaustion function $\rho:X\to\R_+$ and a sequence 
$0<c_0<c_1<\cdots$ with $\lim_{i\to\infty}c_i=+\infty$
such that every $c_i$ is a regular value of $\rho$ and, setting
\[
	D_i=\{x\in X: \rho(x)<c_i\} \quad \text{for $i=0,1,2,\ldots$}, 
\]
we have that $K\subset D_0\subset \overline{D_0} \subset U$, 
where $U\subset X$ is a neighbourhood of $K$ as in the theorem 
(see \cite[Theorem 5.1.6, p. 117]{Hormander1990}). 
Note that the set $\overline {D_i}$ is $\Oscr(X)$-convex for every $i=-1,0,1,\ldots$. 
We may assume that the given holomorphic map $f_0=(z_0,w_0):U\to\Omega$ 
satisfies condition \eqref{eq:estimate} for all $x\in \overline{D_0}$ 
and $z_0(x)\ne 0$ for $x\in bD_0$. 
(We shall use the subscript in $z_i$ and $w_i$ as an index in the induction 
process; a notation for the components of these maps will not be needed.) 
Pick a number $t_0\in\R$ with 
\[
	0<t_0 <\inf_{x\in bD_0}|z_0(x)|.
\]
Choose a sublevel set $D_{-1}=\{\rho<t_{-1}\}$ of $\rho$ such that 
$K\subset D_{-1}\subset \overline{D_{-1}} \subset D_0$ and
\[
	z_0(\overline{D_0\setminus D_{-1}}) \subset \C^{n+1}\setminus t_0\overline\B.
\]
By the Oka--Weil theorem, we may approximate the map 
$w_0:U\to\C^p$ uniformly on $\overline {D_0}$ by a holomorphic map 
$w_1:X\to\C^n$ such that $(z_0,w_1)(\overline{D_0})\subset \Omega$. 

We shall now construct a holomorphic map $z_1:\overline D_1\to\C^{n+1}$
such that the holomorphic map $f_1=(z_1,w_1):\overline D_1 \hra\Omega$ enjoys 
suitable properties to be explained in the sequel. This will be the first step of an 
induction procedure.

Pick a number $t_1\ge t_0+1$ so big that
\begin{equation}\label{eq:t1}
	h(t_1) > \sup \{|w_1(z)|: z\in \overline {D_1}\}.
\end{equation}
(Such a number exists since $\lim_{t\to\infty}h(t)=+\infty$.) 
Fix $\epsilon>0$ whose precise value will be determined later.
Let $\tilde z_0:\overline {D_0} \to\C^{n+1}$ be a holomorphic map
given by Theorem \ref{th:proper} (with $\tilde z_0=g$ in the notation 
of that theorem, applied to the map $g_0=z_0$,
the compact set $\overline{D_{-1}}\subset D_0$, and the 
numbers $\epsilon$ and $t_0<t_1$). Condition (b) in Theorem \ref{th:proper}
gives
\[
	|\tilde z_0(x)| > |z_0(x)|-\epsilon\quad \text{for all $x\in \overline {D_0}$}. 
\]
Since the function $h$ in \eqref{eq:Omega} is continuous, 
it follows that if $\epsilon>0$ is small enough
then the map $(\tilde z_0,w_1):\overline {D_0}\to \C^N$ has range in $\Omega$, 
and we have that 
\begin{equation}\label{eq:tildez0}
	\tilde z_0(bD_0)\subset \C^{n+1}\setminus t_1\overline \B,
	\quad 
	\tilde z_0(\overline {{D_0} \setminus D_{-1}}) 
	\subset \C^{n+1}\setminus t_0\overline \B,
	\quad 
	|\tilde z_0-z_0|<\epsilon\ \ \text{on} \ \overline {D_{-1}}.
\end{equation}

We now use the fact that $\C^{n+1}\setminus t_1\overline \B$ is an Oka domain
(see Kusakabe \cite[Corollary 1.3]{Kusakabe2020complements}).  
Hence, the main result of Oka theory gives 
a holomorphic map $z_1: \overline D_1\to\C^{n+1}$ satisfying
\begin{equation}\label{eq:z1}
	z_1(\overline {D_1 \setminus D_0})
	\subset \C^{n+1}\setminus t_1\overline \B
	\quad \text{and}\quad 
	|z_1-\tilde z_0|<\epsilon\ \ \text{on} \ \overline {D_{0}}.
\end{equation}
(See \cite[Theorem 1.3]{Forstneric2023Indag} for a precise statement of 
a more general result. In the special case at hand, this was proved 
by a more involved argument in the paper 
\cite{ForstnericRitter2014} by Forstneri\v c and Ritter, predating Kusakabe's work 
\cite{Kusakabe2020complements}.) If the number $\epsilon>0$ is chosen small 
enough, it follows from \eqref{eq:t1}--\eqref{eq:z1} and the definition 
of $\Omega$ \eqref{eq:Omega} that 
%\eqref{eq:t1}, \eqref{eq:tildez0}, and \eqref{eq:z1} that 
\begin{equation}\label{eq:f1}
	z_1(\overline {D_0\setminus D_{-1}}) \subset \C^{n+1}\setminus t_0\B 
	\quad\text{and}\quad
	(z_1,w_1)(\overline D_1)\subset \Omega.
\end{equation}
Since the dimension of the target space $\C^N$ is at least $2\dim X+1$,
we may assume after a small perturbation that 
the map $f_1=(z_1,w_1):\overline D_1\to\Omega$ is an embedding 
satisfying the above conditions (see \cite[Corollary 8.9.3]{Forstneric2017E}). 
Assuming as we may that all approximations are close enough,
we also have that $|f_1-f_0|<\epsilon_0$ on $\overline {D_{-1}}$ for a 
given $\epsilon_0>0$.

Continuing inductively, we obtain an increasing 
sequence $t_0 < t_1 <t_2<\cdots$ with $\lim_{i\to\infty} t_i=\infty$,
a decreasing sequence $\epsilon_0>\epsilon_1>\epsilon_2>\cdots>0$ 
with $\lim_{i\to\infty} \epsilon_i=0$, 
and a sequence of holomorphic embeddings 
% F  pogoj v naslednji vrstici se ponovi kot to?ka (i)
$f_i=(z_i,w_i):\overline{D_i} \hra \C^{2n+1}$ 
satisfying the following conditions for $i=1,2,\ldots$.
\begin{enumerate}[\rm (i)]
\item $f_i(\overline D_i)\subset\Omega$.
\item $z_i(\overline{D_i\setminus D_{i-1}}) \subset 
\C^{n+1} \setminus t_{i}\overline\B$.
\item  $z_i(\overline{D_{i-1}\setminus D_{i-2}}) \subset 
\C^{n+1} \setminus t_{i-1}\overline\B$. 
\item $|f_i-f_{i-1}|<\epsilon_{i-1}$ on $\overline{D_{i-2}}$.
\item $t_i\ge t_{i-1}+1$.
\item $0<\epsilon_i<\epsilon_{i-1}/2$.
\item Every holomorphic map
$f:\overline{D_i}\to\C^N$ with $|f-f_i|<2\epsilon_i$ on $\overline D_{i-1}$ 
is an embedding on $\overline{D_{i-2}}$ and 
satisfies $f(\overline {D_{i-1}})\subset \Omega$.
\end{enumerate}
Note that the conditions (i) and (ii) also holds for $i=0$ by the assumptions
on $f_0$, and the conditions (i)--(v) hold for $i=1$ by the construction 
of the map $f_1$. 

The inductive step is similar to the one
from $i=0$ to $i=1$, which was explained above. 
Assume inductively that conditions (i)--(v) hold for some $i\in \{1,2,\ldots\}$.
Pick a number $\epsilon_i$ satisfying conditions (vi) and (vii).
Also, fix a number $\epsilon>0$ whose precise value will be determined
during this induction step. By the Oka--Weil theorem, there is a holomorphic map
$w_{i+1}:X\to\C^p$ with $|w_{i+1}-w_i|<\epsilon$ on $\overline{D_i}$.
Choose a number $t_{i+1}\ge t_i+1$ so big that
\begin{equation}\label{eq:tiplus1}
	h(t_{i+1}) > \sup \{|w_{i+1}(x)|: x \in \overline {D_{i+1}}\}.
\end{equation}
If $\epsilon>0$ is chosen small enough then Theorem \ref{th:proper},
applied to the map $g_0=z_i:\overline{D_i}\to\C^{n+1}$, the compact set
$\overline{D_{i-1}}\subset D_i$, and the numbers $t_i<t_{i+1}$ furnishes
a holomorphic map $\tilde z_i:\overline{D_{i}}\to\C^{n+1}$ 
such that the map $(\tilde z_i,w_{i+1}):\overline {D_0}\to \C^N$ 
has range in $\Omega$ and  
\[ %begin{equation}\label{eq:tildezi}
	\tilde z_i(bD_i)\subset \C^{n+1}\setminus t_{i+1}\overline \B,
	\quad
	\tilde z_i(\overline {{D_i} \setminus D_{i-1}})
	\subset \C^{n+1}\setminus t_i\overline \B,
	\quad
	|\tilde z_i-z_i|<\epsilon\ \ \text{on} \ \overline {D_{i-1}}.	
\] %end{equation}
(For $i=0$ these are conditions \eqref{eq:tildez0}.)
Since $\C^{n+1}\setminus t_{i+1}\overline \B$ is an Oka domain
(see \cite[Corollary 1.3]{Kusakabe2020complements}),
there is a holomorphic map $z_{i+1}: \overline{D_{i+1}} \to\C^{n+1}$ satisfying
\[ %begin{equation}\label{eq:ziplus1}
	z_{i+1}(\overline{{D_{i+1}}\setminus D_i})
	\subset \C^{n+1}\setminus t_{i+1}\overline {\B} 
	\quad\text{and}\quad
	|z_{i+1}-\tilde z_i|<\epsilon\ \ \text{on}\ \overline {D_i}.
\] % end{equation}
(This is an analogue of the condition \eqref{eq:z1}.)
Finally, we perturb the holomorphic map 
\[
	f_{i+1}=(z_{i+1},w_{i+1}):\overline{D_{i+1}} \to\C^N 
\]
slightly to make it an embedding. If all approximations % made in the process 
are close enough then $f_{i+1}$ satisfies conditions (i)--(iv), and (v) 
holds by the choice of $t_{i+1}$. This completes  the induction step.

Conditions (iv) and (vi) imply that the sequence $f_i$ converges to the limit map 
\[
	f=(z,w)=\lim_{i\to\infty} f_i:X\to \C^N
\]
satisfying $|f-f_i|<2\epsilon_i$ on $\overline {D_{i-1}}$ for every $i=0,1,\ldots$. 
(In particular, we have that $|f-f_0|<2\epsilon_0$ on $K$.) 
Conditions (i) and (vii) then imply that $f$ is a holomorphic
embedding with $f(X)\subset \Omega$. Finally, conditions (ii)--(iv) and (vi) 
imply that the map $z:X\to\C^{n+1}$ is proper, and hence $f$ is proper
as map to $\C^N$. % We leave further details to the reader.
\end{proof}

% 
% 
% ACKNOWLEDGEMENTS 
% 
% 
\medskip 
\noindent {\bf Acknowledgements.} 
Forstneri\v c is supported by the European Union (ERC Advanced grant HPDR, 101053085) and grants P1-0291, J1-3005, N1-0237 from ARIS, Republic of Slovenia.

%%%%%%%%%% 
%%%%%%%%%% 
%%%%%%%%%% 
%%%%%%%%%% THE BIBLIOGRAPHY 
%%%%%%%%%% 
%%%%%%%%%% 

%{\bibliographystyle{abbrv} \bibliography{references}} 
%\begin{comment} 

%\end{comment} 

\vspace*{5mm} 
\noindent Franc Forstneri\v c 

\noindent Faculty of Mathematics and Physics, University of Ljubljana, Jadranska 19, SI--1000 Ljubljana, Slovenia 

\noindent 
Institute of Mathematics, Physics and Mechanics, Jadranska 19, SI--1000 Ljubljana, Slovenia 

\noindent e-mail: {\tt franc.forstneric@fmf.uni-lj.si}

\end{document}